\pdfoutput=1 
\documentclass[11pt,a4paper,oneside]{amsart} 

\usepackage{amsmath}
\usepackage{amsfonts}
\usepackage{amssymb}
\usepackage{amsthm}
\usepackage{braket}
\usepackage{mathtools}
\usepackage{array}
\usepackage{enumitem}
\usepackage{mathrsfs}

\usepackage[english]{babel}
\usepackage[utf8]{inputenc}
\usepackage[sc]{mathpazo} %Font Palatino
\usepackage{setspace}
\usepackage{enumitem}
\usepackage{microtype}
\usepackage[usenames,dvipsnames,svgnames,table]{xcolor}
\usepackage[a4paper,top=2cm,bottom=2cm,left=2cm,right=2cm,bindingoffset=5mm]{geometry}

\usepackage{amsmath}
\usepackage{amsfonts}
\usepackage{amssymb}
\usepackage{amsthm}
\usepackage{braket}
\usepackage{mathtools}
\usepackage{array}
\usepackage{enumitem}
\usepackage{mathrsfs}

\theoremstyle{definition}

\theoremstyle{plain}
\newtheorem{prop}{\textbf{Proposition}}[section]
	
\newtheorem{lemma}{\textbf{Lemma}}[section]
\newtheorem{thm}{\textbf{Theorem}}[section]
\theoremstyle{remark}
\newtheorem{rmk}{\textbf{Remark}}[section]

\theoremstyle{plain}

\newtheorem{conjecture}{Conjecture}

\newcommand{\df}{\overset{\text{def}}{=}}

\usepackage[alphabetic]{amsrefs}
\usepackage{hyperref}      

\title{A note on homogeneous ideals of polarized abelian surfaces}
\author{Daniele Agostini}
\address{
  Humboldt-Universit\"{a}t zu Berlin\\
  Institut  f\"{u}r Mathematik\\
  Unter den Linden 6, 10099, Berlin Germany}
\email[]{daniele.agostini@math.hu-berlin.de}

\date{}
\begin{document}
	\maketitle
	
	\begin{abstract}
		Gross and Popescu conjectured that the homogeneous ideal of an embedded $(1,d)$-polarized abelian surface is generated by  quadrics and cubics for $d\geq 9$. We prove a generalization of this using the projective normality of the embedding. It follows that the homogeneous ideal of an abelian surface embedded by a complete linear system is generated by quadrics and cubics, with three exceptions.
	\end{abstract}
	
	\section{Introduction}
	
	We will always work over $\mathbb{C}$, but the same arguments work over an algebraically closed field of characteristic $0$.  
	
	In their paper {\cite{gross_popescu}} Gross and  Popescu proved that if $(A,L)$ is a general polarized abelian surface of type $(1,d)$ with $d\geq 10$, then its homogeneous ideal in the embedding $A\subseteq \mathbb{P}(H^0(A,L))$ is generated by quadrics. At the end of the same paper, they formulated the following conjecture:
	
	\begin{conjecture}{\cite[Conjecture (a)]{gross_popescu}}
		Let $(A,L)$ be a polarized abelian surface of type $(1,d)$ such that $L$ is very ample and $d\geq 9$. Then the homogeneous ideal of $A$ in the embedding $A\subseteq \mathbb{P}(H^0(A,L))$ is generated by  quadrics and cubics. 
	\end{conjecture}
	
	This result was already proven for  $d=7$  by Manolache and Schreyer in {\cite[Corollary 2.2]{manolache_schreyer}}, where they show that the ideal is generated by cubics and compute the whole minimal free resolution.
	The case $d=8$ was proven by Gross and Popescu {\cite[Theorem 6.13]{gross_popescu_cyi}} for a general abelian surface and
	the cases $d\geq 23$ were recently proved by K\"uronya and Lozovanu {\cite[Theorem 1.3]{kuronya_lozovanu}} as a consequence of their Reider-type result for higher syzygies on such surfaces. 
	
	The purpose of this note is to give a simple unified proof of the following:
	
	\begin{prop}\label{proposition}
		Let $(A,L)$ be a polarized abelian surface of type $(1,d)$ such that $L$ is very ample and $d\geq 7$. Then the homogeneous ideal of $A$ in the embedding $A\subseteq \mathbb{P}(H^0(L))$ is generated by  quadrics and cubics. 
	\end{prop}
	
	The proof we give is based on a simple application of Koszul cohomology, together with the following result of  Lazarsfeld  and Fuentes Garc\'ia:
	
	\begin{thm}[Lazarsfeld {\cite{lazarsfeld}}, Fuentes Garc\'ia {\cite{fuentesgarcia}}]\label{projectivenormality}
		Let $(A,L)$ be a polarized abelian surface of type $(1,d)$ such that $L$ is very ample and $d\geq 7$. Then the embedding $A\subseteq \mathbb{P}(H^0(A,L))$ is projectively normal.
	\end{thm}
	
	Proposition \ref{proposition} and previous results give the following:
	
	\begin{thm}\label{theoremideal}
		Let $(A,L)$ be a polarized abelian surfaces with $L$ is very ample and  not of type $(1,5),(1,6),(2,4)$. Then the embedding $A\subseteq \mathbb{P}(H^0(A,L))$ is projectively normal and the homogeneous ideal of $A$ is generated by quadrics and cubics.
	\end{thm}
	
	\textbf{Acknowledgments}: I am very grateful to Robert Lazarsfeld for having made available to me the paper {\cite{lazarsfeld}} and for allowing me to reproduce a part of it here. I would like to thank Ciro Ciliberto and Edoardo Sernesi for helpful discussions. I am happy to thank my advisor Gavril Farkas for his suggestion to study this question and for his advice. I would also like to thank the referee for the helpful comments.

	\section{Koszul cohomology }
	
	Koszul cohomology is a language introduced by Green {\cite{green}} to study minimal free resolutions. We recall here its basic definitions and properties. Let $X$ be an integral projective variety of positive dimension and $L$ a very ample line bundle on $X$. For every vector bundle $E$ on $X$ we can form the group
	\[ \Gamma(X,E,L) = \bigoplus_{q\in \mathbb{Z}} H^0(X,L^{q} \otimes E) \]
	which has a natural structure of a finitely generated $S=\text{Sym}^{\bullet}(H^0(X,L))$-module. In particular we can take its \textit{minimal free resolution}:
	\begin{equation*}
	0 \longleftarrow \Gamma(X,E,L) \longleftarrow F_0 \longleftarrow F_1 \longleftarrow \dots \longleftarrow F_p \longleftarrow \dots F_n \longleftarrow 0
	\end{equation*} 
	where every $F_p$ is a graded free $S$-module:
	\begin{equation*}
	F_p = \bigoplus_{q\in \mathbb{Z}} K_{p,q}(X,E,L) \otimes_{\mathbb{C}} S(-p-q).
	\end{equation*}
	The vector spaces $K_{p,q}(X,E,L)$ are called the \textit{Koszul cohomology groups} of $(X,E,L)$ and they can also be computed as the middle cohomology of the \textit{Koszul complex} (see {\cite[Theorem 1.b.4]{green}}):
	
	\begin{equation*}\label{koszulcomplex} \wedge^{p+1}H^0(X,L) \otimes H^0(X,E \otimes L^{q-1}) \longrightarrow \wedge^p H^0(X,L) \otimes H^0(X,E \otimes L^q) \longrightarrow \wedge^{p-1}H^0(X,L)  \otimes H^0(X,E \otimes L^{q+1}) 
	\end{equation*}
	
	If $E=\mathcal{O}_X$ we denote $K_{p,q}(X,L) : \df K_{p,q}(X,E,L)$.
	
	\begin{rmk}\label{remark}
		From the Koszul complex we see immediately that $K_{p,q}(X,L)=0$ for $q<0$ or $q=0, p>0$. Moreover we also see that the embedding $X\hookrightarrow \mathbb{P}(H^0(X,L))$ is projectively normal if and only if $K_{0,q}(X,L)=0$ for all $q\geq 2$. In this case, the minimal free resolution of the ideal $I_X$ is given by
		\begin{equation*}
		0 \longleftarrow I_X \longleftarrow F_1 \longleftarrow \dots \longleftarrow F_p \longleftarrow \dots F_n \longleftarrow 0
		\end{equation*} 
		so that $I_X$ is generated by quadrics and cubics if and only if the only factors appearing in $F_1$ are $S(-2)$ and $S(-3)$ which means that $K_{1,q}(X,L)=0$ for all $q\geq 3$. 
	\end{rmk}
	
	We will need a duality theorem for Koszul cohomology:
	
	\begin{thm}{\cite[Theorem 2.c.6]{green}}\label{greendualitytheorem}
		Let $X$ be a smooth integral projective variety, $L$ a very ample line bundle on $X$ and $E$ a vector bundle on $X$.  Set $r=h^0(X,L)-1$ and suppose that
		\begin{equation}
		H^i(X,E \otimes L^{q-i}) = H^i(X,E \otimes L^{q-1-i})=0 \qquad \text{ for all } i=1,\dots,\dim X-1
		\end{equation}
		Then there is an isomorphism
		\begin{equation}
		K_{p,q}(X,E,L)^{\vee} \cong K_{r-{\dim X}-p,\dim X+1-q}(X,\omega_X \otimes E^{\vee},L).
		\end{equation}
	\end{thm}
	
	In the case of abelian surfaces, the canonical bundle is trivial, so that we get the following lemma.
	
	\begin{lemma}\label{stupidlemma}
		Let $A\subseteq \mathbb{P}(H^0(L))$ be an abelian surface embedded by a complete linear system. Then if the embedding is projectively normal, the homogeneous ideal of $A$ is generated by quadrics and cubics.
	\end{lemma}
	\begin{proof}
		By Remark \ref{remark}, we just need to prove that $K_{1,q}(A,L)=0$ for all $q\geq 3$. Set $r=h^0(A,L)-1$, then  Theorem \ref{greendualitytheorem} gives an isomorphism
		\[ K_{1,q}(A,L)^{\vee} \cong K_{r-3,3-q}(A,L) \]
		and the group on the right is zero because of Remark \ref{remark}.
	\end{proof}

	\section{Projective normality of polarized abelian surfaces}
	
	The key result that we are going to use is Theorem \ref{projectivenormality}: this was proven by Lazarsfeld {\cite{lazarsfeld}} in the cases $d=7,9,11$ and $d\geq 13$. The remaining cases $d=8,10,12$ were solved by Fuentes Garc\'ia  {\cite{fuentesgarcia}}. 
	
	We would like to sketch here the proof, in particular because the original preprint {\cite{lazarsfeld}} is quite hard to find. We would like to emphasize that all the results and the ideas of this section are due to Lazarsfeld and Fuentes Garc\'ia. Our only contribution is in the presentation of the argument. 
	
	First observe the following:
	\begin{lemma}
		Let $A$ be an abelian surface and $L$ a very ample line bundle on it. Then the embedding $A\subseteq \mathbb{P}(H^0(A,L))$ is projectively normal if and only if the multiplication map
		\[  \text{Sym}^2 H^0(A,L) \longrightarrow H^0(A,L^2) \]
		is surjective.
	\end{lemma}
	\begin{proof}
		This is proven for an abelian variety of any dimension by Iyer in {\cite[Proposition 2.1]{iyer_surfaces}}. In the case of abelian surfaces, we can give a quick proof via Koszul cohomology as follows: suppose that the above multiplication map is surjective, then by definition $K_{0,2}(A,L)=0$ and by Remark \ref{remark} we just need to prove that $K_{0,q}(A,L)=0$ for all $q\geq 3$. In this case we see that $H^1(A,L^{q-1}) = H^1(A,L^{q-2})=0$
		so that Theorem \ref{greendualitytheorem} gives
		\[ K_{0,2}(A,L)^{\vee} \cong K_{r-2,3-q}(A,L) \]
		where $r=h^0(L)-1$. Now, it is enough to observe that $r-2>0$ (since $L$ is very ample) and $3-q\leq 0$ by hypothesis, so Remark \ref{remark} gives $K_{r-2,3-q}(A,L)=0$.
	\end{proof}
	
	Now, let us fix an embedded polarized abelian surface $A \subseteq \mathbb{P}(H^0(A,L))$ with $L$ of type $(1,d)$ and $d\geq 7$. There is an exact sequence
	\begin{equation}\label{multiplicationmap} 0 \longrightarrow I \longrightarrow \text{Sym}^2 H^0(A,L) \longrightarrow H^0(A,L^2) \longrightarrow U \longrightarrow 0 \end{equation}
	and we want to prove $U=0$.  Lazarsfeld {\cite{lazarsfeld}} makes the following observation:
	
	\begin{lemma}\label{boundonu}
		With the above notations, we have $\dim U \leq 6$.
	\end{lemma}
	\begin{proof}
		The argument given here is different from the one in {\cite{lazarsfeld}} and it is a slight clarification of the one given in {\cite{fuentesgarcia_arxiv}}. It is based on the following lemma. 
		
		\begin{lemma}\label{boundonquadrics}
			Let $X\subseteq \mathbb{P}^N$ be a nondegenerate integral surface of degree $t$. Then
			\[ \dim H^0(\mathbb{P}^N,\mathscr{I}_{X/\mathbb{P}^N}(2)) \leq \frac{N(N-1)}{2} - \min\{t,2N-5\} \]
		\end{lemma}
		\begin{proof}
			First recall that $t\geq N-1$ (see {\cite[Proposition 0]{eisenbud_harris}}). Now choose $H\subseteq \mathbb{P}^N$ to be a general linear subspace of codimension 2. Then $H\cap X$ consists of $t$ distinct points in linearly general position in $H$: in particular, they span $H$, since $t\geq \dim H +1$. 
			
			Now we observe that there is no quadric $Q\subseteq \mathbb{P}^N$ containing both $X$ and $H$. Indeed, suppose that there is such a $Q$: then, since $X$ is nondegenerate it would have rank at least $3$ so that its singular locus $\text{Sing}(Q)$ is a linear subspace of codimension at least 3, which cannot contain  $H\cap X$. This shows that $X\cap H\cap (Q\setminus \text{Sing}(Q)) \ne \emptyset$ and since $H\cap (Q\setminus \text{Sing}(Q))$ is a Cartier divisor on $Q\setminus \text{Sing}(Q)$ it follows from Krull's principal ideal theorem that every irreducible component of $ X\cap H \cap (Q\setminus \text{Sing}(Q))$ has positive dimension, which gives a contradiction.
			
			This shows that the restriction map
			\[ H^0(\mathbb{P}^N,\mathscr{I}_{X/\mathbb{P}^N}(2)) \longrightarrow H^0(H,\mathscr{I}_{X\cap H /H}(2)) \]
			is injective. To conclude, we can just apply Castelnuovo's argument for which $t\geq N-1$ points in linearly general position in $\mathbb{P}^{N-2}$ impose at least $\min \{t,2N-5\}$ independent conditions on quadrics (see {\cite[Lemma p.115]{acgh}}).  
		\end{proof}
		
		Now it is immediate to prove Lemma \ref{boundonu}: we have $A\subseteq \mathbb{P}(H^0(A,L))\cong \mathbb{P}^{d-1}$ of degree $2d$, so that applying Lemma \ref{boundonquadrics} we obtain
		\[ \dim U = \dim H^0(A,L^2) - \dim \text{Sym}^2H^0(A,L) + \dim I \leq 4d - \binom{d+1}{2} + \frac{d(d-7)}{2} + 6 = 6   \]
	\end{proof}
	
	Then it is enough to show that if $U\ne 0$, then $\dim U \geq 7$. Lazarsfeld's idea is to use the Heisenberg group associated to $(A,L)$.
	
	\subsection{The Heisenberg group}
	
	Recall that to $(A,L)$ we can associate the two groups $K(L)=\{x\in X | t_x^*L \cong L\}$ and $G(L)=\{(\alpha,x)|\alpha\colon L \to t_x^*L \text{ isomorphism }\}$. The group $G(L)$ is called the \textit{Heisenberg group} of $L$ and it is a central extension of $K(L)$ by $\mathbb{C}^*$ {\cite[Theorem 1]{mumford_equations}}:
	\[ 1 \longrightarrow \mathbb{C}^* \longrightarrow G(L) \longrightarrow K(L) \longrightarrow 0 \]
	A linear representation of $G(L)$ where $\mathbb{C}^*$ acts by the character $\lambda \mapsto \lambda^k$ is called a\textit{ representation of weight $k$}. The space $H^0(A,L)$ has a natural linear action of $G(L)$ given by $(\alpha,x)\cdot \sigma = t_{-x}^{*}(\alpha(\sigma))$ and, up to isomorphism, this is the unique irreducible representation of $G(L)$ of weight 1 (see {\cite[Proposition 3, Theorem 2]{mumford_equations}}). This representation induces other representations of weight $2$ on $\text{Sym}^2 H^0(A,L)$ and $H^0(A,L^2)$ such that the multiplication map (\ref{multiplicationmap}) is $G(L)$-equivariant. In particular, $I$ and $U$ can be regarded as $G(L)$-representations of weight $2$. The irreducible ones have been classified by Iyer {\cite[Proposition 3.2]{iyer_surfaces}}.
	
	\begin{prop}[Iyer]\label{irredrepr}
		Let $(A,L)$ be a polarized abelian surface of type $(1,d)$. Then 
		\begin{enumerate}
			\item if $d$ is odd, there is, up to isomorphism, a unique irreducible $G(L)$-representation of weight $2$. This representation has dimension $d$.
			\item if $d=2m$ is even, then there are, up to isomorphism, four distinct $G(L)$-representations of weight $2$. Each irreducible representation has dimension $m$.	
		\end{enumerate}	
	\end{prop}	
	
	This proves Theorem \ref{projectivenormality} for most cases:
	
	\begin{proof}[Proof of Theorem \ref{projectivenormality}]
		Suppose that $d$ is odd and greater than $7$ or even and greater than $14$. Assume by contradiction that in \ref{multiplicationmap} we have $U\ne 0$: then, since $U$ is a $G(L)$-representation it must be by Proposition \ref{irredrepr} that $\dim U \geq 7$. However, this is impossible because $\dim U \leq 6$ by Lemma \ref{boundonu}.	
		
		This leaves the cases $d=8,10,12$, which were solved by Fuentes Garc\'ia in {\cite{fuentesgarcia}}  using the involutions in $G(L)$ coming from the $2$-torsion points of $K(L)$, together with geometric results about polarized abelian surfaces of small degree.
	\end{proof}
	
	\begin{rmk}\label{projectivenormalitykuronyalozovanu}
		Theorem \ref{projectivenormality} can be proven for $d\geq 10$ also using the results of K\"uronya and Lozovanu: indeed in this case $(L)^2 \geq 20$, so that by {\cite[Theorem 1.1]{kuronya_lozovanu}} the embedding is not projectively normal if and only if there exists an elliptic curve $E\subseteq A$ such that $(E)^2=0$ and $(E\cdot L)=1,2$. In particular, $L_{|E}$ is not very ample so that $L$ cannot be very ample.
	\end{rmk}	
	
	\section{Homogeneous ideals of polarized abelian surfaces}
	Now it is very easy to prove Proposition \ref{proposition}:
	
	\begin{proof}[Proof of Proposition \ref{proposition}]  Follows immediately from Theorem \ref{projectivenormality} and Lemma \ref{stupidlemma}.
	\end{proof}
	
	Theorem \ref{theoremideal} instead is a consequence of the following:
	
	\begin{thm}[Koizumi {\cite{koizumi}}, Lazarsfeld {\cite{lazarsfeld}}, Fuentes Garc\'ia {\cite{fuentesgarcia}}, Ohbuchi {\cite{ohbuchi}}]\label{classificationprojectivenormality}
		Let $A\subseteq \mathbb{P}(H^0(L))$ be an abelian surface embedded by a complete linear system. Then the embedding is projectively normal, unless $L$ is of type $(1,5),(1,6),(2,4)$. In these cases, the embedding is never projectively normal.
	\end{thm}
	
	\begin{proof}
		Suppose that $L$ is of type $(d_1,d_1 m)$: if $d_1\geq 3$, then the result was first proven  by Koizumi \cite{koizumi}. Another proof can be found in {\cite[Theorem 7.3.1]{birkenhake_lange}}. 
		
		If $d_1=2$ and $m\geq 3$ then projective normality follows from a result by Ohbuchi \cite{ohbuchi}. Alternatively,  we can reason as in Remark \ref{projectivenormalitykuronyalozovanu}.  Ohbuchi also shows in \cite[Lemma 6]{ohbuchi} that if  $m=2$, then  $L$ is not projectively normal. For another proof of this, Barth has shown in \cite[Theorem 2.11]{barth} that the ideal $I_{A}$ contains precisely 6 linearly independent quadrics. Hence $\text{Sym}^2 H^0(A,L) \longrightarrow H^0(A,L^2)$ has image of dimension $36-6=30$, which is less than the dimension of $H^0(A,L^2)$.
		
		If $d_1=1$, then this is Theorem \ref{projectivenormality}. Observe that there cannot be projective normality for $L$ of type $(1,5)$ or $(1,6)$  because in these cases $\text{Sym}^2H^0(A,L)$ has dimension smaller than $H^0(A,L^2)$.
		
		In all the other cases, it is easy to see that the line bundle cannot be very ample.
	\end{proof}
	
	\begin{proof}[Proof of Theorem \ref{theoremideal}]
		This follows immediately from Theorem \ref{classificationprojectivenormality}, Lemma \ref{stupidlemma}.
	\end{proof}
	
	\begin{rmk}
		We can also say something about the exceptional cases: for a very ample line bundle of type $(1,5)$ Manolache has proven {\cite[Theorem 1]{manolache}} that the homogeneous ideal is generated by 3 quintics and 15 sextics. For the case $(1,6)$  Gross and Popescu {\cite[Remark 4.8.(2)]{gross_popescu_cyi}} have proven that the ideal sheaf of a general such abelian surface is generated by cubics and quartics. For the case $(2,4)$ Barth {\cite[Theorem 2.14,Theorem 4.9]{barth}} gives explicit quadrics which generate the ideal sheaf of the surface: it is then easy (for example with Macaulay2 {\cite{M2}}) to compute examples where the homogeneous ideal is generated by quadrics and quartics. \end{rmk}

\end{document}